\documentclass[10pt,a4paper]{amsart}
\usepackage{graphicx,multirow,array,amsmath,amssymb}

\newtheorem{theorem}{Theorem}

\newtheorem{lemma}[theorem]{Lemma}

\begin{document}

\title{Upper bound on the total number of knot $n$-mosaics}

\author[K. Hong]{Kyungpyo Hong}
\address{Department of Mathematics, Korea University, Anam-dong, Sungbuk-ku, Seoul 136-701, Korea}
\email{cguyhbjm@korea.ac.kr}
\author[H. Lee]{Ho Lee}
\address{Department of Mathematical Sciences, KAIST, 291 Daehak-ro, Yuseong-gu, Daejeon 305-701, Korea}
\email{figure8@kaist.ac.kr}
\author[H. J. Lee]{Hwa Jeong Lee}
\address{Department of Mathematical Sciences, KAIST, 291 Daehak-ro, Yuseong-gu, Daejeon 305-701, Korea}
\email{hjwith@kaist.ac.kr}
\author[S. Oh]{Seungsang Oh}
\address{Department of Mathematics, Korea University, Anam-dong, Sungbuk-ku, Seoul 136-701, Korea}
\email{seungsang@korea.ac.kr}

\thanks{The corresponding author(Seungsang Oh) was supported by Basic Science Research Program through
the National Research Foundation of Korea(NRF) funded by the Ministry of Science,
ICT \& Future Planning(MSIP) (No.~2011-0021795).}
\thanks{This work was supported by the National Research Foundation of Korea(NRF) grant
funded by the Korea government(MEST) (No. 2011-0027989).}

\begin{abstract}
Lomonaco and Kauffman introduced a knot mosaic system to give a definition of a quantum knot system
which can be viewed as a blueprint for the construction of an actual physical quantum system.
A knot $n$-mosaic is an $n \times n$ matrix of 11 kinds of specific mosaic tiles representing a knot or a link
by adjoining properly that is called suitably connected.
$D_n$ denotes the total number of all knot $n$-mosaics.
Already known is that $D_1=1$, $D_2=2$, and $D_3=22$.
In this paper we establish the lower and upper bounds on $D_n$ 
$$\frac{2}{275}(9 \cdot 6^{n-2} + 1)^2 \cdot 2^{(n-3)^2} \ \leq \ D_n \
\leq \ \frac{2}{275}(9 \cdot 6^{n-2} + 1)^2 \cdot (4.4)^{(n-3)^2}.$$
and find the exact number of $D_4 = 2594$.
\end{abstract}

\maketitle

\section{Introduction} 

Knot theory and other areas of topology have made propound impact on quantum field theory, 
quantum computation and complexity of computation.
Lomonaco and Kauffman introduced a knot mosaic system to set the foundation
for a quantum knot system in the series of papers \cite{K, L, LK1, LK2, LK3, LK4}.
This paper was inspired from an open question about the enumeration of knot mosaics in \cite{LK2}.

Throughout this paper we will frequently use the term ``knot" to mean either a knot or a link
for simplicity of exposition.
Let $\mathbb{T}$ denote the set of the following $11$ symbols which are called {\em mosaic tiles\/}; 

\begin{figure}[ht]
\includegraphics{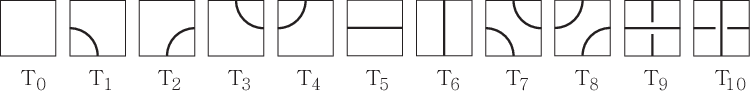}
\label{fi1}
\end{figure}

For a positive integer $n$,
we define an {\em $n$-mosaic\/} as an $n \times n$ matrix $M=(M_{ij})$ of mosaic tiles.
We denote the set of all $n$-mosaics by $\mathbb{M}^{(n)}$.
Obviously $\mathbb{M}^{(n)}$ has $11^{n^2}$ elements.
A {\em connection point\/} of a tile is defined as the midpoint of a mosaic tile edge
which is also the endpoint of a curve drawn on the tile.
Then each tile has zero, two or four connection points as follows; 

\begin{figure}[ht]
\includegraphics{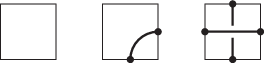}
\label{fi2}
\end{figure}

We say that two tiles in a mosaic are {\em contiguous\/} if they lie immediately next to each other
in either the same row or the same column.
A mosaic tile within a mosaic is said to be {\em suitably connected\/}
if each of its connection points touches a connection point of a contiguous tile.
A {\em knot $n$-mosaic\/} is an $n$-mosaic in which all tiles are suitably connected.
Then a knot $n$-mosaic represents a specific knot.
In Figure \ref{fig1}, we draw three examples of mosaics; a 4-mosaic, the Hopf link 4-mosaic
and the trefoil knot 4-mosaic.

\begin{figure}[ht]
\includegraphics{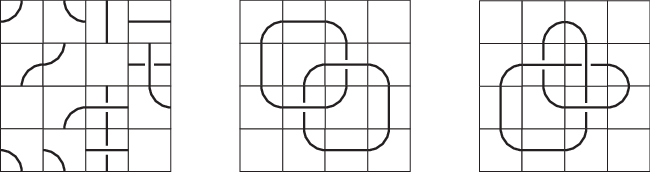}
\caption{Three examples of 4-mosaics}
\label{fig1}
\end{figure}

As an analog to the planar isotopy moves and the Reidemeister moves for standard knot diagrams,
Lomonaco and Kauffman created for knot mosaics the $11$ mosaic planar isotopy moves
and the mosaic Reidemeister moves in \cite{LK2}.
They conjectured that
for any two tame knots (or links) $K_1$ and $K_2$,
and their arbitrary chosen mosaic representatives $M_1$ and $M_2$, respectively,
$K_1$ and $K_2$ are of the same knot type if and only if
$M_1$ and $M_2$ are of the same knot mosaic type.
This means that tame knot theory and knot mosaic theory are equivalent.
Kuriya and Shehab \cite{KS} proved that Lomonaco-Kauffman conjecture is true.

Lomonaco and Kauffman also proposed a dozen of open questions relevant to quantum knot mosaics.
One natural question is how many knot $n$-mosaics are there.
Let $\mathbb{K}^{(n)}$ denote the subset of $\mathbb{M}^{(n)}$ of all knot $n$-mosaics,
and $D_n$ the total number of elements of $\mathbb{K}^{(n)}$.
The main theme in this paper is to establish upper and lower bounds on $D_n$.
Already known is that $D_1=1$, $D_2=2$ and $D_3=22$, 
for which the complete table of $\mathbb{K}^{(3)}$ is in Appendix A in \cite{LK2}.
One might gave a very loose upper bound $11^{n^2}$.

\begin{theorem} \label{thm:main}
For an integer $n \geq 3$,
$$\frac{2}{275}(9 \cdot 6^{n-2} + 1)^2 \cdot 2^{(n-3)^2} \ \leq \ D_n \
\leq \ \frac{2}{275}(9 \cdot 6^{n-2} + 1)^2 \cdot (4.4)^{(n-3)^2}.$$
\end{theorem}

\begin{theorem} \label{thm:D4}
$D_4 = 2594$.
\end{theorem}

Recently, the authors announced several improved results on $D_n$ in the series of papers.
They concerned the exact number of $D_n$ for small $n=4,5,6$ \cite{HLLO}, 
the state matrix algorithm, so called, producing the exact enumeration of general $D_n$
that uses recursion formula of state matrices \cite{OHLL},
and more precise bounds of the quadratic exponential growth ratio of $D_n$ \cite{O}.

Another interesting question relevant to knot mosaics is the 
{\em mosaic number\/} $m(K)$ of a knot $K$ as the smallest integer $n$
for which $K$ is representable as a knot $n$-mosaic.
Is this mosaic number related to the crossing number of a knot?
As an concrete answer, the authors \cite{LHLO} established an upper bound on the mosaic number as follows;
if $K$ be a nontrivial knot or a non-split link except the Hopf link,
then $m(K) \leq c(K) + 1$,
and moreover if $K$ is prime and non-alternating except $6^3_3$, then $m(K) \leq c(K) - 1$.
Note that the mosaic number of the Hopf link is 4,
and the prime and non-alternating $6^3_3$ link is 6,
even though their crossing numbers are 2 and 6, respectively.

\section{Proof of Theorem \ref{thm:main}}

For $n \geq 3$, $\mathbb{K}^{(n)}$ is the set of knot $n$-mosaics,
so each mosaic is filled by suitably connected $n^2$ mosaic tiles entirely.
$\mathbb{K}^{(n)}_1$ denotes the set of so called {\em $n$-quasimosaics\/} each of which is filled by
suitably connected $2n-3$ mosaic tiles only at $M_{1j}$ and $M_{i1}$, $i,j=1,2, \cdots , n-1$.
It is indeed a part of a knot $n$-mosaic, and so possibly has connection points on the boundary
contained in the interior of the knot $n$-mosaic.
Similarly $\mathbb{K}^{(n)}_2$ denotes the set of $n$-quasimosaics each of which is filled by
suitably connected $4n-8$ tiles at $M_{1j}$, $M_{2j}$, $M_{i1}$, and $M_{i2}$, $i,j=1,2, \cdots , n-1$.
Also $\mathbb{K}^{(n)}_3$ denotes the set of $n$-quasimosaics each of which is filled by
suitably connected $(n-1)^2$ tiles at $M_{ij}$, $i,j=1,2, \cdots , n-1$.
Let $d_1$, $d_2$ and $d_3$ denote the numbers of elements of
$\mathbb{K}^{(n)}_1$, $\mathbb{K}^{(n)}_2$ and $\mathbb{K}^{(n)}_3$, respectively.
See three typical examples of elements of $\mathbb{K}^{(6)}_1$, $\mathbb{K}^{(6)}_2$ and
$\mathbb{K}^{(6)}_3$ in Figure \ref{fig2}.

\begin{figure}[ht]
\includegraphics{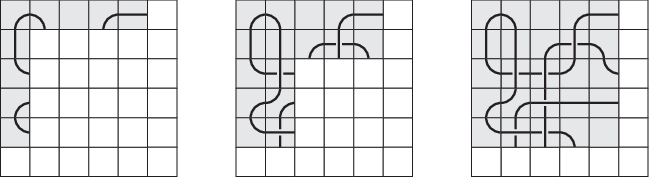}
\caption{Three elements of $\mathbb{K}^{(6)}_1$, $\mathbb{K}^{(6)}_2$ and $\mathbb{K}^{(6)}_3$}
\label{fig2}
\end{figure}

For simplicity of exposition, a mosaic tile is called {\em $t$-cp\/} if it has a connection point
on its top edge,
and similarly {\em $b$-, $l$-\/} or {\em $r$-cp\/} when on its bottom, left or right edge, respectively.
Sometimes we use two letters, for example, $tl$-cp in the case of both $t$-cp and $l$-cp.
We use the sign $\hat{}$ \/ for negation such as
$\hat{t}$-cp means not $t$-cp,
$\hat{t} \hat{l}$-cp means both $\hat{t}$-cp and $\hat{l}$-cp, and
$\widehat{tl}$-cp (which is differ from $\hat{t} \hat{l}$-cp) means not $tl$-cp,
i.e. $\hat{t} l$-, $t \hat{l}$- or $\hat{t} \hat{l}$-cp.

First we figure out $\mathbb{K}^{(n)}_1$ and determine the number $d_1$.

\begin{lemma} \label{lem:d1}
$d_1 = 2^{2n-3}$.
\end{lemma}

\begin{proof}
We use the induction on $n$.
The first mosaic tile $M_{11}$ has 2 choices whether $T_0$ or $T_2$.
The next tile $M_{12}$ has always 2 choices after any choices of $M_{11}$ as follows;
if $M_{11} =T_0$, then $M_{12}$ is $\hat{t} \hat{l}$-cp,
so $M_{12}$ is either $T_0$ or $T_2$,
or if $M_{11} =T_2$, then $M_{12}$ must be $\hat{t} l$-cp to be suitably connected,
so $M_{12}$ is either $T_1$ or $T_5$.
By the same reason each $M_{1j}$, $j=3, \cdots , n-1$, has always 2 choices;
if $M_{1(j-1)}$ is $\hat{r}$-cp, then $\hat{t}\hat{l}$-cp $M_{1j}$ is either $T_0$ or $T_2$,
or if $M_{1(j-1)}$ is $r$-cp, then $\hat{t}l$-cp $M_{1j}$ is either $T_1$ or $T_5$.
We can follow the same argument when we choose mosaic tiles $M_{i1}$, $i=2, \cdots , n-1$.
Thus if $M_{(i-1)1}$ is $\hat{b}$-cp,
then $\hat{t}\hat{l}$-cp $M_{i1}$ is either $T_0$ or $T_2$,
or if $M_{(i-1)1}$ is $b$-cp, then $t\hat{l}$-cp $M_{i1}$ is either $T_3$ or $T_6$.
Therefore each tile has exactly 2 choices.
Since each $n$-quasimosaic of $\mathbb{K}^{(n)}_1$ consists of $2n-3$ mosaic tiles,

$d_1 = 2^{2n-3}$.
\end{proof}

\noindent {\bf Fact 1.\/} For any $j=2, \cdots , n-1$, exactly the half of $\mathbb{K}^{(n)}_1$
have $b$-cp $M_{1j}$'s and the rest half have $\hat{b}$-cp $M_{1j}$'s.
Similarly for any $i=2, \cdots , n-1$, exactly the half of $\mathbb{K}^{(n)}_1$
have $r$-cp $M_{i1}$'s and the rest half have $\hat{r}$-cp $M_{i1}$'s.
\vspace{3mm}

\noindent {\bf Fact 2.\/} For any $i,j=2, \cdots , n-1$,
$M_{ij}$ is one of $T_4$, $T_7$, $T_8$, $T_9$ or $T_{10}$ if it is $tl$-cp,
either $T_1$ or $T_5$ if $\hat{t} l$-cp,
either $T_3$ or $T_6$ if $t \hat{l}$-cp, and
either $T_0$ or $T_2$ if $\hat{t} \hat{l}$-cp.
Therefore each $M_{ij}$ has 5 choices of mosaic tiles if it is $tl$-cp,
and 2 choices if it is $\widehat{tl}$-cp.
\vspace{3mm}

Next we figure out $\mathbb{K}^{(n)}_2$ and determine the number $d_2$.

\begin{lemma} \label{lem:d2}
$d_2 = \frac{2}{275} (9 \cdot 6^{n-2} + 1)^2$.
\end{lemma}

\begin{proof}
Similar to the definitions of $\mathbb{K}^{(n)}_2$ and $d_2$,
let $\mathbb{K}^{(n)}_{2j}$, $j=2, \cdots , n-1$, denote the set of all $n$-quasimosaics
each of which is filled by suitably connected mosaic tiles as in $\mathbb{K}^{(n)}_1$
and more tiles at $M_{2k}$, $k=2, \cdots , j$.
Let $d_{2j}$ denote the number of elements of $\mathbb{K}^{(n)}_{2j}$.

First we fill the mosaic tile $M_{22}$.
By Fact 1, exactly $(\frac{1}{2})^2 d_1$ elements of $\mathbb{K}^{(n)}_1$
have $tl$-cp $M_{22}$'s to be suitably connected,
and the rest $\frac{3}{4} d_1$ elements have $\widehat{tl}$-cp $M_{22}$'s.
By Fact 2, $d_{22} = \frac{1}{4} d_1 \cdot 5 + \frac{3}{4} d_1 \cdot 2 = \frac{11}{4} d_1$.
Note that among all $d_{22}$ elements of $\mathbb{K}^{(n)}_{22}$,
$\frac{1}{4} d_1 \cdot 4 + \frac{3}{4} d_1 \cdot 1 = \frac{7}{4} d_1 = \frac{7}{11} d_{22}$
elements have $r$-cp $M_{22}$'s.
Let $p_2 = \frac{7}{11}$.

Now we use the induction again.
For any $j=3, \cdots , n-1$,
the same argument above guarantees that
exactly $\frac{1}{2} p_{j-1} \cdot d_{2(j-1)}$ elements of $\mathbb{K}^{(n)}_{2(j-1)}$
can be suitably connected with $tl$-cp $M_{2j}$'s,
and the rest elements with $\widehat{tl}$-cp $M_{2j}$'s.
Thus $d_{2j} = \frac{1}{2} p_{j-1} \cdot d_{2(j-1)} \cdot 5 +
(1 - \frac{1}{2} p_{j-1}) \cdot d_{2(j-1)} \cdot 2 =
(2 + \frac{3}{2} p_{j-1}) \cdot d_{2(j-1)}$.
Then among all $d_{2j}$ elements of $\mathbb{K}^{(n)}_{2j}$,
$\frac{1}{2} p_{j-1} \cdot d_{2(j-1)} \cdot 4 +
(1 - \frac{1}{2} p_{j-1}) \cdot d_{2(j-1)} \cdot 1 =
(1 + \frac{3}{2} p_{j-1}) \cdot d_{2(j-1)} =
\frac{2 + 3 p_{j-1}}{4 + 3 p_{j-1}} \cdot d_{2j}$
elements have $r$-cp $M_{2j}$'s.
Let $p_j = \frac{2 + 3 p_{j-1}}{4 + 3 p_{j-1}}$.

Therefore $d_{2(n-1)} = d_1 \cdot \frac{11}{4} \cdot
(2 + \frac{3}{2} p_2) \cdots (2 + \frac{3}{2} p_{n-2})$.
Since $p_j = \frac{2 \cdot 6^j - 2}{3 \cdot 6^j + 2}$ satisfies the recurrence relation for \{$p_j$\},
we have the equation $2 + \frac{3}{2} p_j =
\frac{1}{2} \cdot \frac{3 \cdot 6^{j+1} +2}{3 \cdot 6^j +2}$.
To fill all the tiles (especially on the second row and the second column)
of elements of $\mathbb{K}^{(n)}_2$;

$d_2 = d_1 \cdot \frac{11}{4} \cdot
(2 + \frac{3}{2} p_2)^2 \cdots (2 + \frac{3}{2} p_{n-2})^2
=\frac{2}{275} (9 \cdot 6^{n-2} + 1)^2$.
\end{proof}

Now we figure out $\mathbb{K}^{(n)}_3$ and find bounds on $d_3$.

\begin{lemma} \label{lem:d3}
$\frac{2}{275} (9 \cdot 6^{n-2} + 1)^2 \cdot 2^{(n-3)^2} \leq
d_3 \leq \frac{2}{275} (9 \cdot 6^{n-2} + 1)^2 \cdot (4.4)^{(n-3)^2}$.
\end{lemma}

\begin{proof}
Let $i,j=3, \cdots , n-1$.
As a continuation of Fact 2,
if $M_{ij}$ is $tl$-cp, then four tiles $T_7$, $T_8$, $T_9$, and $T_{10}$ among 5 choices
have $r$-cp,
or if $M_{ij}$ is $\widehat{tl}$-cp, then one tile among 2 choices has $r$-cp.
This fact guarantees that between one-half and four-fifths quasimosaics
of $\mathbb{K}^{(n)}_3$ have $r$-cp $M_{ij}$'s, and similarly for $b$-cp $M_{ij}$'s.

Unlike the argument in the proof of Lemma \ref{lem:d2},
the two probabilities of $M_{ij}$ having $t$-cp and $l$-cp are not independent.
To calculate $d_3$, we thus have to multiply to $d_2$
at least $0 \cdot 5 + (1 - 0) \cdot 2 = 2$ and
at most $\frac{4}{5} \cdot 5 + (1 - \frac{4}{5}) \cdot 2 = 4.4$ for each $M_{ij}$.
Thus we have;

$d_2 \cdot 2^{(n-3)^2} \leq d_3 \leq d_2 \cdot (4.4)^{(n-3)^2}$.
\end{proof}

Finally we will finish the proof of Theorem \ref{thm:main}.
For each $n$-quasimosaic of $\mathbb{K}^{(n)}_3$, there is exactly one way to fill
mosaic tiles to be suitably connected at every $M_{nj}$ or $M_{in}$
where $i,j = 1, \cdots , n$,
because every tile has even numbered connection points.
This implies that $D_n = d_3$.

Indeed the inequality of the upper bound appears only on Lemma \ref{lem:d3}.
This means that the equality holds for $n=3$, so $D_3=22$.

\section{$D_4 = 2594$}

In this section we consider $\mathbb{K}^{(4)}$ and find the exact number $D_4$.
$\mathbb{K}^{(4)}_c$ denotes the set of all $4$-quasimosaics each of which is filled by
suitably connected 4 mosaic tiles only at $M_{ij}$, $i,j=2,3$.
Let $d_c$ denote the number of elements of $\mathbb{K}^{(4)}_c$.
A common edge of two $M_{ij}$'s is called a {\em central edge\/}.
Note that there are four central edges as bold segments depicted in Figure \ref{fig3}.
\vspace{3mm}

\begin{figure}[ht]
\includegraphics{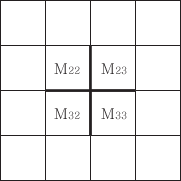}
\caption{Four central edges}
\label{fig3}
\end{figure}

\noindent {\bf Fact 3.\/} As in Fact 2,
if both central edges of $M_{ij}$ have connection points,
then $M_{ij}$ has 5 choices of mosaic tiles.
Otherwise, it has 2 choices.
\vspace{3mm}

First we figure out $\mathbb{K}^{(4)}_c$ and find the number $d_c$.
Since each central edge has 2 cases whether it has a connection point or not,
we split into 16 cases whether each of four central edges has a connection point or not.

Among 16 cases,
there is only one case where all four central edges have connection points.
By Fact 3, every $M_{ij}$ has 5 choices, so we have $5^4$ different $4$-quasimosaics in $\mathbb{K}^{(4)}_c$.
There are four cases where exactly three central edges have connection points.
In each case two of $M_{ij}$'s have 5 choices and the other two have 2 choices,
and so we have $5^2 \cdot 2^2$ different $4$-quasimosaics.
There are another four cases where
only two perpendicular central edges have connection points.
In each case only one of $M_{ij}$'s has 5 choices and the other three have 2 choices,
and so we have $5 \cdot 2^3$.
In each of the rest seven cases,
every $M_{ij}$ has 2 choices, so we have $2^4$.
Thus we have the following;

$d_c = 5^4 + 4 \cdot 5^2 \cdot 2^2 + 4 \cdot 5 \cdot 2^3 + 7 \cdot 2^4 = 1297$.

Finally we are ready to finish the proof of Theorem \ref{thm:D4}.
For each $4$-quasimosaic in $\mathbb{K}^{(4)}_c$, there are exactly two ways to fill
mosaic tiles to be suitably connected at the rest twelve boundary $M_{ij}$'s.
For, every tile has even numbered connection points,
so the union of boundary edges of $M_{22} \cup M_{23} \cup M_{32} \cup M_{33}$
has even number of connection points.
This implies that $D_4 = 2 d_c$.

\end{document}